\documentclass[11pt]{amsart}
\usepackage{mathrsfs}
\usepackage{amsfonts}
\usepackage{latexsym,amsmath,amssymb}
 \renewcommand{\epsilon}{\varepsilon}

\newtheorem{theorem}{Theorem}[section]
 
 \newtheorem{lemma}[theorem]{Lemma}
 
 \newtheorem{Corollary}[theorem]{Corollary}
 \newtheorem{proposition}[theorem]{proposition}
 \newtheorem{Proposition}[theorem]{Proposition}
\newtheorem{deff}[theorem]{Definition}
 \newtheorem{rem}[theorem]{Remark}
 \newcommand{\bth}{\begin{theorem}}
 \newcommand{\ble}{\begin{lemma}}
 \newcommand{\bcor}{\begin{corr}}
 \newcommand{\bdeff}{\begin{deff}}
 \newcommand{\bprop}{\begin{proposition}}
 \newcommand{\ele}{\end{lemma}}
 \newcommand{\ecor}{\end{corr}}
 \newcommand{\edeff}{\end{deff}}
 
 \newcommand{\eprop}{\end{proposition}}

 \renewcommand{\Pi}{\varPi}

 \renewcommand{\epsilon}{\varepsilon}

\numberwithin{equation}{section}

\title
[Nonlinear second boundary conditions ]{On the second boundary value problem for Lagrangian mean curvature flow}
\author{Rongli Huang}
\address{School of Mathematics and Statistics, Guangxi Normal University,
Guilin, Guangxi 541004, People's Republic of China,
 E-mail: ronglihuangmath@gxnu.edu.cn}
\date{}

\begin{document}
\maketitle
%\section{}
%\subsection{}

\begin{abstract}
We consider a fully nonlinear parabolic equation with nonlinear Neumann type boundary condition, and
show the long time existence and convergence of the flow. Finally we apply this study to the boundary value problem
for minimal Lagrangian graphs.

\end{abstract}

\let\thefootnote\relax

\section{Introduction}
Since the work of R.P.Thomas and S.T.Yau
\cite{RY}   about mean curvature flow of Lagrangian submainfolds of Calabi-Yau manifolds,
Lagrangian mean curvature flow has been studying by many authors. K. Smoczyk and M.T. Wang  obtained the long time existence and convergence of  Lagrangian mean curvature flow
in some conditions  (cf.\cite{K1}, \cite{KM}). The progress on singularity of Lagrangian mean curvature flow made people
have a deeper understanding to Thomas-Yau Conjectures such as J.Y. Chen and J.Y. Li \cite{CL}, A. Neves \cite{A1} \cite{A2}.
Recently several authors took the equation point of view to study  Lagrangian mean curvature flow
such as \cite{ACH1}, \cite{ACY1}. \par
Let  $\Omega$, $\tilde{\Omega}$  be strict convex bounded  domains with smooth boundary in $\mathbb{R}^{n}$.
In  special Lagrangian geometry,
S. Brendle and M. Warren \cite{SM} used the method of continuity to prove that there exists a diffeomorphism f: $\Omega\rightarrow\tilde{\Omega}$ such that
the graph $$\Sigma=\{(x,f(x))|x\in \Omega\}$$
is a minimal Lagrangian submanifold of $\mathbb{R}^{n}\times\mathbb{R}^{n}$. The aim of this paper is to provide a parabolic approach to  Brendle-Warren' theorem.

Firstly we introduce some relevant works according to solving elliptic equations with second boundary conditions by parabolic approach. To solve an optimal transportation,  J. Kitagawa\cite{JK} looked for solutions on the following
set of boundary value problems:
\begin{equation*}\label{e1.6}
\left\{ \begin{aligned}\frac{\partial u}{\partial t}-\ln\det (D^{2}u-A(x, Du))&=-\ln B(x,Du)
& t>0,\quad x\in \Omega, \\
Du(\Omega)&=\tilde{\Omega}, &t>0,\qquad\qquad\\
 u&=u_{0}, & t=0,\quad x\in \Omega.
\end{aligned} \right.
\end{equation*}
Here $A$ is a matrix value function and $B$ a scalar value function defined on the cost function and
two measures related to the transportation. Under certain conditions on $\Omega, \tilde{\Omega}$, $A$, $B$ and the initial
function, he proved the long time existence to the above flow, and convergence to the solution of the optimal transport
problem as $t\rightarrow +\infty$.
In\cite{OK}, Neumann and second boundary value problems for Hessian and  Gauss curvature flows were carefully studied by O.C. Schnurer and K. Smoczyk.
They showed that the flow exists for all time and converges eventually to the solution of the
prescribed Gauss curvature equation.

Inspired from \cite{JK} and \cite{OK},
 we consider the following Lagrangian mean curvature flow with boundary conditions
\begin{equation}\label{e1.1}
\left\{ \begin{aligned}\frac{\partial u}{\partial t}&=\Sigma_{i=1}^{n}\arctan\lambda_{i},
& t>0,\quad x\in \Omega, \\
Du(\Omega)&=\tilde{\Omega}, &t>0,\qquad\qquad\\
 u&=u_{0}, & t=0,\quad x\in \Omega.
\end{aligned} \right.
\end{equation}
where
$$\lambda_{1}\leq\lambda_{2}\leq\cdots\leq\lambda_{n}$$
are the eigenvalues of $D^{2}u=[u_{ij}]$,  and $Du(\cdot,t)$ is
a family of diffeomorphisms from $\Omega$ to $\tilde{\Omega}$.
Along the lines of approach in  a work by O.C. Schnurer and K. Smoczyk \cite{OK}, our main results concern the long time existence and convergence of the nonlinear parabolic flow (\ref{e1.1}) and then obtain the solution to the boundary value problem for minimal lagrangian graphs \cite{SM}. Now we can state our main theorem.

\begin{theorem}\label{t1.1}
Assume that $\Omega$, $\tilde{\Omega}$ are bounded, strictly convex domains with smooth boundary in $\mathbb{R}^{n}$ and $0<\alpha<1$.
  Then for any given initial function $u_{0}\in C^{2+\alpha}(\bar{\Omega})$
  which is   strictly convex and satisfies $Du_{0}(\Omega)=\tilde{\Omega}$,  the  strictly convex solution of (\ref{e1.1}) exists
  for all $t\geq 0$ and $u(\cdot,t)$ converges to a function $u^{\infty}$ in $C^{1+\zeta}(\bar{\Omega})\cap C^{\infty}(\bar{D})$ as $t\rightarrow\infty$
  for any $D\subset\subset\Omega$, $\zeta<1$,
  and $u^{\infty}\in C^{1+1}(\bar{\Omega})\cap C^{\infty}(\Omega)$ is a solution of
\begin{equation}\label{e1.2}
\left\{ \begin{aligned}\Sigma_{i=1}^{n}\arctan\lambda_{i}&=c,
&  x\in \Omega, \\
Du(\Omega)&=\tilde{\Omega}.
\end{aligned} \right.
\end{equation}
The constant $c$ depends only on $\Omega$, $\tilde{\Omega}$ and $u_{0}$.
\end{theorem}
\begin{rem}
By the methods in \cite{SM}, the initial function $u_{0}$  can be obtained by considering
\begin{equation*}\label{e}
\left\{ \begin{aligned}\triangle u&=c,
&  x\in \Omega, \\
Du(\Omega)&=\tilde{\Omega}.
\end{aligned} \right.
\end{equation*}
Here the goal is easier to attack because Laplace equation is simpler than special Lagrangian equation.
\end{rem}
\begin{rem}
S. Brendle and M. Warren \cite{SM} showed that the solution to (\ref{e1.2}) was unique up to addition of constants.
\end{rem}

It's well known that (\ref{e1.2}) is special Lagrangian eqaution with  second boundary condition where the solution $(x,Du)$ is a minimal Lagrangian graph in $\mathbb{R}^{n}\times\mathbb{R}^{n}$.
Using the method of solving fully nonlinear elliptic equations with second boundary conditions, S. Brendle and M. Warren \cite{SM} obtained the existence and uniqueness of the solution to (\ref{e1.2}). As a consequence of Theorem \ref{t1.1}, we proved the existence result of the minimal Lagrangian submanifolds with the same conditions in $\mathbb{R}^{n}\times\mathbb{R}^{n}$.

The plan is as follows for the proof of Theorem 1.1. In Section 2, we establish the local existence result to the flow (\ref{e1.1}) by
the inverse function theory. In Section 3, we provide preliminary results which will be used in the proof of the theorem. The techniques used in this section are reflective of those in \cite{JU} and \cite{OK} to the second boundary value problems for fully nonlinear differential equations, but all of the corresponding a priori estimates to the solution in the current scenario
need modification because the structure of (\ref{e1.1})  is unlike Monge-Amp$\grave{e}$re type.  In Section 4, we give the proof of our main results.

\section{ the short-time existence of the
parabolic flow }

 Throughout the following Einstein's convention of
 summation over repeated indices will be adopted.
Denote $$u_{i}=\dfrac{\partial u}{\partial x_{i}},
u_{ij}=\dfrac{\partial^{2}u}{\partial x_{i}\partial x_{j}},
u_{ijk}=\dfrac{\partial^{3}u}{\partial x_{i}\partial x_{j}\partial
x_{k}}, \cdots ,$$ and
$$[u^{ij}]=[ u_{ij}]^{-1},\,\,\,F(D^{2}u)=\Sigma_{i=1}^{n}\arctan\lambda_{i},\,\,\,
F^{ij}(D^{2}u)=\frac{\partial F}{\partial u_{ij}},\,\,\, \Omega_{T}=\Omega\times(0,T). $$
By the methods on the second boundary value problems for equations of Monge-Amp\`{e}re type \cite{JU},
the parabolic boundary condition in (\ref{e1.1}) can be reformulated as
$$h(Du)=0,\qquad x\in \partial\Omega,\quad t>0,$$
where $h$ is a smooth function on $\bar{\tilde{\Omega}}$:
$$\tilde{\Omega}=\{p\in\mathbb{R}^{n} |h(p)>0\},\qquad |Dh|_{{\partial\tilde{\Omega}}}=1.$$
The so called boundary defining function is strictly concave, i.e, $\exists \theta>0$,
$$ \frac{\partial^{2}h}{\partial y_{i}\partial y_{j}}\xi_{i}\xi_{j}\leq
-\theta|\xi|^{2},\qquad \mathbf{for}\quad \forall y=(y_{1}, y_{2},\cdots, y_{n})\in \tilde{\Omega},\quad
 \xi=(\xi_{1}, \xi_{2},\cdots, \xi_{n})\in \mathbb{R}^{n}.$$
We also give the boundary defining function according to $\Omega$ (cf.\cite{SM}:
$$\Omega=\{p\in\mathbb{R}^{n} |\tilde{h}(p)>0\},\qquad |D\tilde{h}|_{{\partial\Omega}}=1,$$
$$\exists\tilde{\theta}>0,\quad\frac{\partial^{2}\tilde{h}}{\partial y_{i}\partial y_{j}}\xi_{i}\xi_{j}\leq
-\tilde{\theta}|\xi|^{2},\quad \mathbf{for}\quad \forall y=(y_{1}, y_{2},\cdots, y_{n})\in \tilde{\Omega},\quad
 \xi=(\xi_{1}, \xi_{2},\cdots, \xi_{n})\in \mathbb{R}^{n}.$$
Thus the parabolic flow is equivalent to the evolution problem:
\begin{equation}\label{e1.3}
\left\{ \begin{aligned}\frac{\partial u}{\partial t}&=\Sigma_{i=1}^{n}\arctan\lambda_{i},
& t>0,\quad x\in \Omega, \\
h(Du)&=0,& \qquad t>0,\quad x\in\partial\Omega,\\
 u&=u_{0}, & \qquad\quad t=0,\quad x\in \Omega.
\end{aligned} \right.
\end{equation}
To obtain the short-time existence of classical solution of (\ref{e1.3}) we use the inverse function theorem
in Fr$\acute{e}$chet spaces and the theory of linear parabolic equations for oblique boundary conditions.
\begin{lemma}[\cite{IE}, Theorem 2]\label{l1.1}
Let $X$ and $Y$ be Banach spaces. Denote $$J: X\rightarrow Y$$
 be continuous and G$\hat{a}$teaux-differentiable, with
$J(v_{0})=w_{0}$. Assume that the derivative $DJ[v]$ has a right inverse $L[v]$, uniformly bounded in
a neighbourhood of $v_{0}$:
$$\forall \alpha\in Y,\quad DJ[v]L[v]\alpha=\alpha$$
$$\|v-v_{0}\|\leq R\Longrightarrow \|L[v]-L[v_{0}]\|\leq m.$$
For every $w\in Y$  if
$$\parallel w-w_{0}\parallel<\frac{R}{m}$$
then there is some $v$ such that we have:
$$\|v-v_{0}\|< R,$$
and
$$J(v)=w.$$
\end{lemma}
\begin{lemma}[\cite{GM}, Theorem 8.8 and 8.9]\label{l1.2}
Assume that $f\in C^{\alpha,\frac{\alpha}{2}}(\bar{\Omega}_{T})$  for some $0<\alpha<1$, $T>0$, and
$G(x,p)$,$G_{p}(x,p)$ are in $C^{1+\alpha}(\Sigma)$ for any compact subset $\Sigma$ of $\partial\Omega\times\mathbb{R}^{n}$
such that $\inf_{\partial\Omega}\langle G_{p}, \nu\rangle>0$
where $\nu$ is the inner normal vector of $\partial\Omega$. Let $u_{0}\in C^{2+\alpha}(\bar{\Omega})$ be strictly convex and satisfies
$G(x, Du_{0})=0.$ Then there exists  $T_{max}>0$ such that we can find an unique solution which is strictly convex in $x$ variable in the class $C^{2+\alpha,1+\frac{\alpha}{2}}(\bar{\Omega}_{T_{max}})$
to the following equations
\begin{equation*}\label{e1.4}
\left\{ \begin{aligned}\frac{\partial u}{\partial t}-\triangle u&=f(x,t),
& T>t>0,\quad x\in \Omega, \\
G(x,Du)&=0,& \qquad T>t>0,\quad x\in\partial\Omega.\\
 u&=u_{0}, & \qquad\quad t=0,\quad x\in \Omega.
\end{aligned} \right.
\end{equation*}
\end{lemma}
According to the proof of \cite{JU}, one can verify the oblique boundary condition.
\begin{lemma}[J. Urbas\cite{JU}]\label{l1.3}\quad
\\
$u\in C^{2}(\bar{\Omega})$  with $D^{2}u>0$
$\Longrightarrow$ $\inf_{\partial\Omega}h_{p_{k}}(Du)\nu_{k}>0$ where $\nu=(\nu_{1},\nu_{2}, \cdots,\nu_{n})$ is the unit inward normal vector of $\partial\Omega$,
i.e. $h(Du)=0$ is strictly oblique.
\end{lemma}
We are now in a position to prove the short-time existence of solutions of (\ref{e1.3})
which is equivalent to the problem (\ref{e1.1}).
\begin{Proposition}\label{p1.1}
According to the conditions in Theorem \ref{t1.1}, there exists some $T_{max}>0$ and $u\in C^{2+\alpha,1+\frac{\alpha}{2}}(\bar{\Omega}_{T_{max}})$ which depend only on $\Omega$, $\tilde{\Omega}$, $u_{0}$,
 such that $u$ is  a solution of (\ref{e1.3}) and  is strictly convex in $x$ variable.
\end{Proposition}

\begin{proof}
Denote the Banach spaces
$$X=C^{2+\alpha,1+\frac{\alpha}{2}}(\bar{\Omega}_{T}),
\quad Y=C^{\alpha,\frac{\alpha}{2}}(\bar{\Omega}_{T})
\times C^{1+\alpha,\frac{1+\alpha}{2}}(\partial\Omega\times(0,T))\times C^{2+\alpha}(\bar{\Omega}),$$
where $$\parallel\cdot\|_{Y}=\parallel\cdot\|_{C^{\alpha,\frac{\alpha}{2}}(\bar{\Omega}_{T})}
+\parallel\cdot\|_{C^{1+\alpha,\frac{1+\alpha}{2}}(\partial\Omega\times(0,T))}+
\parallel\cdot\|_{C^{2+\alpha}(\bar{\Omega})}.$$
Define a map $$J:\quad X\rightarrow Y$$ by
$$J(u)=\left\{ \begin{aligned}
&\frac{\partial u}{\partial t}-F(D^{2}u), &\quad (x,t)\in\Omega_{T}, \\
&h(Du), &\quad (x,t)\in \partial\Omega\times(0,T),\\
&u, &\quad (x,t)\in \Omega\times\{t=0\}.
\end{aligned} \right.$$
The strategy is now to use the inverse function theorem to obtain the local existence result.
The computation of the G$\hat{a}$teaux derivative shows that:
$$\forall u,v\in X,\quad DJ[u](v)\triangleq\frac{d}{d\tau}J(u+\tau v)|_{\tau=0}=\left\{ \begin{aligned}
&\frac{\partial v}{\partial t}-F^{ij}(D^{2}u)v_{ij}, &\quad (x,t)\in\Omega_{T}, \\
&h_{p_{i}}(Du)v_{i}, &\quad (x,t)\in \partial\Omega\times(0,T),\\
&v, &\quad (x,t)\in \Omega\times\{t=0\}.
\end{aligned} \right.$$
Using Lemma \ref{l1.2} there exists $T_{max}>0$ such that  we can find $\hat{u}\in X$ to be strictly convex in $x$ variable,  which satisfies the following equations :
\begin{equation}\label{e2.3}
\left\{ \begin{aligned}\frac{\partial \hat{u}}{\partial t}-\triangle \hat{u}&=F(D^{2}u_{0})-\triangle u_{0},
& T_{max}>t>0,\quad x\in \Omega, \\
h(D\hat{u})&=0,& \qquad T_{max}>t>0,\quad x\in\partial\Omega.\\
 \hat{u}&=u_{0}, & \qquad\quad t=0,\quad x\in \Omega.
\end{aligned} \right.
\end{equation}
For each $(f,g,w)\in Y $, using Lemma \ref{l1.2} again  there exists an unique $v\in X$ satisfying $DJ[\hat{u}](v)=(f,g,w)$, i.e.,
\begin{equation*}
\left\{ \begin{aligned}\frac{\partial v}{\partial t}-F^{ij}(D^{2}\hat{u})v_{ij}&=f,
& T_{max}>t>0,\quad x\in \Omega, \\
h_{p_{i}}(D\hat{u})v_{i}&=g,& \qquad T_{max}>t>0,\quad x\in\partial\Omega.\\
 v&=w, & \qquad\quad t=0,\quad x\in \Omega.
\end{aligned} \right.
\end{equation*}
Then  the derivative $DJ[\hat{u}]$ has a right inverse $L[\hat{u}]$ and for $T=T_{max}$ we see that
\begin{equation}\label{e2.4}
\forall \gamma=(f,g,w)\in Y,\quad DJ[\hat{u}]L[\hat{u}]\gamma=\gamma.
\end{equation}
If  set $$\hat{f}=\frac{\partial \hat{u}}{\partial t}-F(D^{2}\hat{u}),\quad w_{0}=(\hat{f}, 0,u_{0}),\quad
w=(0, 0,u_{0})$$
then one can show that
\begin{equation*}
\begin{aligned}
\parallel\hat{f}-0\parallel_{C^{\alpha,\frac{\alpha}{2}}(\bar{\Omega}_{T_{max}})}
&=\parallel\triangle \hat{u}-\triangle u_{0}+F(D^{2}u_{0})-F(D^{2}\hat{u})\parallel_{C^{\alpha,\frac{\alpha}{2}}(\bar{\Omega}_{T_{max}})}\\
&\leq\parallel\triangle \hat{u}-\triangle u_{0}\parallel_{C^{\alpha,\frac{\alpha}{2}}(\bar{\Omega}_{T_{max}})}+\parallel F(D^{2}u_{0})-F(D^{2}\hat{u})\parallel_{C^{\alpha,\frac{\alpha}{2}}(\bar{\Omega}_{T_{max}})}\\
&\leq C\parallel \hat{u}-u_{0}\parallel_{C^{2+\alpha,1+\frac{\alpha}{2}}(\bar{\Omega}_{T_{max}})}
\end{aligned}
\end{equation*}
where $C$ is a constant depending only on the known data. We may apply (\ref{e2.3}) to conclude:
$\forall \epsilon>0,$  $\exists T_{max}>0$ to be small enough such that
$$\parallel\hat{f}-0\parallel_{C^{\alpha,\frac{\alpha}{2}}(\bar{\Omega}_{T_{max}})}\leq  C\parallel \hat{u}-u_{0}\parallel_{C^{2+\alpha,1+\frac{\alpha}{2}}(\bar{\Omega}_{T_{max}})}<\epsilon.$$
Thus it's obtained
$$\parallel w-w_{0}\|_{Y}=
\parallel0-\hat{f}\|_{C^{\alpha,\frac{\alpha}{2}}(\bar{\Omega}_{T_{max}})}<\epsilon.
$$
Combining with (\ref{e2.4}) and using Lemma \ref{l1.1}  it gives the desired results.
\end{proof}
\begin{rem}
By the strong maximum principle, the strictly convex solution to (\ref{e1.3}) is unique.
\end{rem}

\section{ Preliminary results }
In this section, the $C^{2}$ a priori bound is accomplished by making the second derivative estimates on the boundary
for solution of parabolic type special lagrangian equation. This treatment
is similar to the problems presented in \cite{JK}, \cite{OK} and \cite{JU}, but requires some modification to
accommodate the particular  situation. Specifically, Corollary \ref{c3.3} is needed in order to
drive differential inequalities from  barriers  which can be used.

For the convenience, we set
$$[g^{ij}]\triangleq[\frac{\partial F(D^{2}u)}{\partial u_{ij}}]=[\delta_{ij}+u_{ik}u_{kj}]^{-1},\quad
\beta^{k}\triangleq\frac{\partial h(Du)}{\partial u_{k}}=h_{p_{k}}(Du)$$
and  $\langle\cdot,\cdot\rangle$ be the inner product in $\mathbb{R}^{n}$.
By Proposition \ref{p1.1} and the regularity theory of parabolic equations,   we may assume that $u$ is a strictly convex solution of (\ref{e1.3}) in the class $C^{2+\alpha,1+\frac{\alpha}{2}}(\bar{\Omega}_{T})\cap C^{\infty}(\Omega_{T})$ for some $T=T_{max}>0$.
\begin{lemma}[$\dot{u}$-estimates]\label{l3.1}\quad
\\
As long as the convex solution to (\ref{e1.3}) exists, the following estimates hold, i.e.
\begin{equation*}\label{e3.1}
0\leq\dot{u}\triangleq\frac{\partial u}{\partial t}\leq\Theta_{0}\triangleq\max_{\bar{\Omega}}F(D^{2}u_{0}).
\end{equation*}
\end{lemma}
\begin{proof}
We use the methods known from Lemma 2.1 in \cite{OK}. \par
From (\ref{e1.3}), a direct computation shows that
$$\frac{\partial\dot{u} }{\partial t}-g^{ij}\partial_{ij}\dot{u}=0.$$
Using the  maximum principle we see that
$$\max_{\bar{\Omega}_{T}}\dot{u} =\max_{\partial\bar{\Omega}_{T}}\dot{u}.$$
Without loss of generality, we assume that $\dot{u}\neq constant$. If $\exists x\in \partial\Omega, t>0,$ such that $\dot{u}(x,t)=\max_{\bar{\Omega}_{T}}\dot{u}.$
Then we differentiate the boundary condition and obtain
$$\dot{u}_{\beta}=\frac{\partial h(Du)}{\partial t}=0.$$
Since $\langle\beta, \nu\rangle>0$, it contradicts  the Hopf Lemma (cf.\cite{LL}) for parabolic equations.
So that
$$\dot{u}\leq \max_{\bar{\Omega}_{T}}\dot{u}
=\max_{\partial\bar{\Omega}_{T}\mid_{t=0}}\dot{u}=\max_{\bar{\Omega}}F(D^{2}u_{0}).$$
On the other hand, $u$ is  convex
$\Longrightarrow\min_{\bar{\Omega}}F(D^{2}u)\geq0\Longrightarrow\dot{u}=F(D^{2}u)\geq 0.$
Putting these facts together, the assertion follows.
\end{proof}
Since $-\frac{\pi}{2}\leq\arctan\lambda_{i}\leq\frac{\pi}{2}$,
 $\arctan\lambda_{i}=\frac{\pi}{2}\Leftrightarrow \lambda_{i}=+\infty$. Then $\Theta_{0}<\frac{n\pi}{2}$.
\begin{lemma}\label{l3.2} Let $(x,t)$  be arbitrary point of $\Omega_{T}$, and $\lambda_{1}\leq\lambda_{2}\leq\cdots\leq\lambda_{n}$
be the eigenvalues of $D^{2}u$ at $(x,t)$. Then
\begin{equation}\label{e3.2}
0\leq\lambda_{1}\leq \tan(\frac{\Theta_{0}}{n}).
\end{equation}
\end{lemma}
\begin{proof}
It follows from the definition of $F(D^{2}u)$ and Lemma \ref{l3.1}:
$$n\arctan\lambda_{1}\leq \Sigma_{i=1}^{n}\arctan\lambda_{i}=\dot{u}\leq \Theta_{0}.$$
Combining with the convexity of $u$ we obtain
$$0\leq \arctan\lambda_{1}\leq \frac{\Theta_{0}}{n}$$
which yields (\ref{e3.2}).
\end{proof}
Now we can show the operator $F$ to be uniformly elliptic which will play an important role
in the barrier arguments.
\begin{Corollary}\label{c3.3}
For any $(x,t)\in\Omega_{T},$ we have
\begin{equation*}\label{e3.3}
\frac{1}{1+\tan(\frac{\Theta_{0}}{n})^{2}}\leq\Sigma_{i=1}^{n}g^{ii}\leq n.
\end{equation*}
\end{Corollary}
\begin{proof}
We observe
$$\Sigma_{i=1}^{n}g^{ii}=\Sigma_{i=1}^{n}\frac{1}{1+\lambda_{i}^{2}}.$$
By Lemma \ref{l3.2}, we obtain
$$\frac{1}{1+\tan(\frac{\Theta_{0}}{n})^{2}}\leq\frac{1}{1+\lambda_{1}^{2}}\leq\Sigma_{i=1}^{n}g^{ii}
=\Sigma_{i=1}^{n}\frac{1}{1+\lambda_{i}^{2}}\leq n.$$
\end{proof}
Returning to Lemma \ref{l1.3}, using Corollary \ref{c3.3} we can get a uniform positive lower bound
of the quantity $\inf_{\partial\Omega}h_{p_{k}}(Du)\nu_{k}$ which does not depend on $t$.
\begin{lemma}\label{l3.4}
As long as the uniformly convex solution to (\ref{e1.3}) exists,  the strict oblique estimates
can be obtained by
\begin{equation}\label{e3.4}
\langle\beta, \nu\rangle\geq \frac{1}{C_{1}}>0,
\end{equation}
where the constant $C_{1}$ is independent of $t$.
\end{lemma}
\begin{proof}
Let $(x_{0},t_{0})\in \partial\Omega\times[0,T]$ such that
$$\langle\beta, \nu\rangle(x_{0},t_{0})=h_{p_{k}}(Du)\nu_{k}=\min_{\partial\Omega\times[0,T]}\langle\beta, \nu\rangle.$$
By the computation in \cite{JU} it gives
\begin{equation}\label{e3.0}
\langle\beta, \nu\rangle=\sqrt{u^{ij}\nu_{i}\nu_{j}h_{p_{k}}h_{p_{l}}u_{kl}}.
\end{equation}
Further on, we may assume that $t_{0}>0$ and $\nu(x_{0})=(0,0,\cdots,1)\triangleq e_{n}$. As in
the proof of Lemma 8.1 in \cite{OK}, by the convexity of $\Omega$ and its smoothness, we extend $\nu$ smoothly to a tubular neighborhood of $\partial\Omega$ such that in matrix sense
\begin{equation}\label{e3.5}
D_{k}\nu_{l}\equiv \nu_{kl}\leq -\frac{1}{C}\delta_{kl}
\end{equation}
for some positive constant $C$. Let
$$v=\langle\beta, \nu\rangle+h(Du).$$
By the above assumptions and the boundary condition, it's obtained  $$v(x_{0},t_{0})=\min_{\partial\Omega\times[0,T]}v
=\min_{\partial\Omega\times[0,T]}\langle\beta, \nu\rangle.$$
In $(x_{0},t_{0})$,  we have
\begin{equation}\label{e3.6}
0=v_{r}=h_{p_{n}p_{k}}u_{kr}+ h_{p_{k}}v_{kr}+h_{p_{k}}u_{kr},\quad 1\leq r\leq n-1,
\end{equation}
\begin{equation*}\label{e3.7}
0\leq \dot{v}.
\end{equation*}
We assume that the following key estimate holds which will be proved later,
\begin{equation}\label{e3.8}
v_{n}(x_{0},t_{0})\geq -C,
\end{equation}
where $C$ is a constant depending only on $\Omega$, $u_{0}$, $h$, and $\tilde{h}$. It's not hard to check that
 (\ref{e3.8}) can be rewritten as
\begin{equation}\label{e3.9}
h_{p_{n}p_{k}}u_{kn}+ h_{p_{k}}\nu_{kn}+h_{p_{k}}u_{kn}\geq -C.
\end{equation}
Multiplying (\ref{e3.9}) with $h_{p_{n}}$ and (\ref{e3.6}) with $h_{p_{r}}$ respectively,
and summing up together  we obtain:
\begin{equation}\label{e3.10}
h_{p_{k}}h_{p_{l}}u_{kl}\geq -Ch_{p_{n}}-h_{p_{k}}h_{p_{l}}\nu_{kl}-h_{p_{k}}h_{p_{n}p_{l}}u_{kl}.
\end{equation}
By the concavity of $h$, we have
$$-h_{p_{n}p_{n}}\geq 0, \quad h_{p_{k}}u_{kr}=\frac{\partial h(Du)}{\partial x_{r}}=0,\quad
h_{p_{k}}u_{kn}=\frac{\partial h(Du)}{\partial x_{n}}=\frac{\partial h(Du)}{\partial x_{n}}\geq 0.$$
Substituting those into (\ref{e3.10}) and using (\ref{e3.5}) it yields
\begin{equation*}
h_{p_{k}}h_{p_{l}}u_{kl}\geq -Ch_{p_{n}}+\frac{1}{C}|Dh|^{2}=-Ch_{p_{n}}+\frac{1}{C}.
\end{equation*}
According to the above last term, we distinguish two cases.

Case (i).
$$ -Ch_{p_{n}}+\frac{1}{C}\leq 0.$$
Then
$$h_{p_{k}}(Du)\nu_{k}=h_{p_{n}}\geq\frac{1}{C^{2}}.$$
It shows that there is a uniform positive lower bound
of the quantity $\min_{\partial\Omega\times[0,T]}h_{p_{k}}(Du)\nu_{k}$.

Case (ii).
$$ -Ch_{p_{n}}(x_{0})+\frac{1}{C}>0.$$
Then we obtain a positive lower bound of $h_{p_{k}}h_{p_{l}}u_{kl}$. Introduce the Legendre transformation of $u$,
\begin{equation*}
y_{i}=\frac{\partial u}{\partial x_{i}}
,\,\,i=1,2,\cdots,n,\,\,\,u^{*}(y_{1},\cdots,y_{n},t):=\sum_{i=1}^{n}x_{i}\frac{\partial u}{\partial x_{i}}-u(x,t).
\end{equation*}
In terms of $y_{1},\cdots,y_{n}, u^{*}(y_{1},\cdots,y_{n})$, one can easily check that
$$\frac{\partial^{2} u^{*}}{\partial y_{i}\partial y_{j}}=[\frac{\partial^{2} u}{\partial x_{i}\partial x_{j}}]^{-1}.$$
Since  $\arctan\lambda+\arctan\lambda^{-1}=\frac{\pi}{2}$.
Then $u^{*}$ satisfies
\begin{equation}\label{e3.12}
\left\{ \begin{aligned}\frac{\partial u^{*}}{\partial t}-F(D^{2}u^{*})&=-\frac{n\pi}{2},
& T>t>0,\quad x\in \tilde{\Omega}, \\
\tilde{h}(Du^{*})&=0,& \qquad T>t>0,\quad x\in\partial\tilde{\Omega},\\
 u^{*}&=u^{*}_{0}, & \qquad\quad t=0,\quad x\in \tilde{\Omega}.
\end{aligned} \right.
\end{equation}
where $\tilde{h}$ is a smooth and strictly concave function on $\bar{\Omega}$:
$$\Omega=\{p\in\mathbb{R}^{n} |\tilde{h}(p)>0\},\qquad |D\tilde{h}|_{{\partial\tilde{\Omega}}}=1.$$
We also define
$$\tilde{v}=\tilde{\beta}^{k}\tilde{\nu}_{k}+\tilde{h}(Du^{*})=\langle\tilde{\beta}, \tilde{\nu}\rangle+\tilde{h}(Du^{*}),$$
where$$\tilde{\beta}^{k}\triangleq\frac{\partial \tilde{h}(Du^{*})}{\partial u^{*}_{k}}=\tilde{h}_{p_{k}}(Du^{*}),$$
and  $\tilde{\nu}=(\tilde{\nu}_{1}, \tilde{\nu}_{2},\cdots,\tilde{\nu}_{n})$ is the inner unit normal
 vector of $\partial\tilde{\Omega}$.
Using the same methods, under the assumption of
\begin{equation*}\label{e3.13}
\tilde{v}_{n}(y_{0},t_{0})\geq -C,
\end{equation*}
we obtain the positive lower bounds of $\tilde{h}_{p_{k}}\tilde{h}_{p_{l}}u^{*}_{kl}$
or
$$h_{p_{k}}(Du)\nu_{k}=\tilde{h}_{p_{k}}(Du^{*})\tilde{\nu}_{k}(y_{0})=\tilde{h}_{p_{n}}\geq\frac{1}{C^{2}}.$$
We notice that
$$\tilde{h}_{p_{k}}\tilde{h}_{p_{l}}u^{*}_{kl}=\nu_{i}\nu_{j}u^{ij}.$$
Then the claim follows from (\ref{e3.0}) by the positive lower bounds of $h_{p_{k}}h_{p_{l}}u_{kl}$ and $\tilde{h}_{p_{k}}\tilde{h}_{p_{l}}u^{*}_{kl}$.

It remains to prove the key estimate (\ref{e3.8}). The proof of Lemma 8.1 in \cite{OK} can
be also adapted  here. For the convenience of readers and the completeness, we provide
the details and the arguments below.

Define the linearized operator by
$$L=g^{ij}\partial_{ij}-\partial_{t}.$$
Since $D^{2}\tilde{h}\leq-\tilde{\theta}I$,  we obtain
\begin{equation}\label{e3.151}
\aligned
L\tilde{h} &\leq  -\tilde{\theta}\sum g^{ii}.\\
\endaligned
\end{equation}
On the other hand,
\begin{equation*}\aligned
Lv=&h_{p_{k}p_{l}p_{m}}\nu_{k}g^{ij}u_{li}u_{mj}+2h_{p_{k}p_{l}}g^{ij}\nu_{kj}u_{li}
+h_{p_{k}p_{l}}g^{ij}u_{lj}u_{ki}\\
&+h_{p_{k}p_{l}}\nu_{k}Lu_{l}+h_{p_{k}}L\nu_{k}.
\endaligned
\end{equation*}
By estimating the first term in the diagonal basis, one yields
$$\mid h_{p_{k}p_{l}p_{m}}\nu_{k}g^{ij}u_{li}u_{mj}\mid\leq C\sum\frac{\lambda^{2}_{i}}{1+\lambda^{2}_{i}}\leq C,$$
where $C$ is a constant depending only on $h$ and $\Omega$. For the same reason, we have
$$\mid2h_{p_{k}p_{l}}g^{ij}\nu_{kj}u_{li}\mid\leq C,\quad \mid h_{p_{k}p_{l}}g^{ij}u_{lj}u_{ki}\mid\leq C.$$
After the simple calculation it gives
$$Lu_{l}=0.$$
Obviously we have
$$|h_{p_{k}}L\nu_{k}|\leq C\sum g^{ii}.$$
So there exists a positive constant $C$ such that
\begin{equation}\label{e3.15}
\mid Lv\mid\leq C\sum g^{ii}
\end{equation}
Here we use  Corollary \ref{c3.3} and $C$ depends  only on $h$, $\Omega$ and $u_{0}$.

Denote a neighborhood of $x_{0}$:
$$\Omega_{\delta}\triangleq \Omega\cap B_{\delta}(x_{0})$$
where $\delta$ is a positive constant such that $\nu$ is well defined in $\Omega_{\delta}$.
We consider
$$\Phi\triangleq v(x,t)-v(x_{0},t_{0})+C_{0}\tilde{h}(x)+A|x-x_{0}|^{2}$$
where $C_{0}$ and $A$ are positive constants to be determined. On $\partial\Omega\times[0,T)$
it is clear that $\Phi\geq 0.$ Since $v$ is bounded, we can select $A$
large enough such that
\begin{equation*}
\begin{aligned}
&(v(x,t)-v(x_{0},t_{0})+C_{0}\tilde{h}(x)+A|x-x_{0}|^{2})|_{(\Omega\cap\partial B_{\delta}(x_{0}))\times[0,T]}\\
 &\geq v(x,t)-v(x_{0},t_{0})-C_{0}C+A\delta^{2}\\
 &\geq 0.
\end{aligned}
\end{equation*}
Using the strictly concavity of $\tilde{h}$ we have
$$ \triangle(C_{0}\tilde{h}(x)+A|x-x_{0}|^{2})\leq C(-C_{0}\tilde{\theta}+2A)\sum g^{ii}.$$
Then by choosing the constant $C_{0}\gg A$,
we can show that
\begin{equation*}
\begin{aligned}
\triangle(v(x,0)-v(x_{0},t_{0})+C_{0}\tilde{h}(x)+A|x-x_{0}|^{2})\leq 0.
\end{aligned}
\end{equation*}
It follows from the maximum principle:
\begin{equation*}
\begin{aligned}
&(v(x,0)-v(x_{0},t_{0})+C_{0}\tilde{h}(x)+A|x-x_{0}|^{2})|_{\Omega_{\delta}}\\
&\geq\min_{(\partial\Omega\cap B_{\delta}(x_{0}))\cup(\Omega\cap\partial B_{\delta}(x_{0})} (v(x,0)-v(x_{0},t_{0})+C_{0}\tilde{h}(x)+A|x-x_{0}|^{2})
\\
&\geq 0.
\end{aligned}
\end{equation*}
Combining (\ref{e3.151}) with (\ref{e3.15}), letting $C_{0}$ be large enough  we obtain $$L\Phi\leq (-C_{0}\tilde{\theta}+C+2A)\sum g^{ii}\leq 0.$$
 From the above arguments one can verify that
$\Phi$ satisfies
\begin{equation}\label{e3.16}
\left\{ \begin{aligned}L\Phi&\leq 0,\qquad
&(x,t)\in\Omega_{\delta}\times[0,T] , \\
\Phi&\geq 0,\qquad &(x,t)\in(\partial\Omega_{\delta}\times[0,T]\cup(\Omega_{\delta}\times\{t=0\}.
\end{aligned} \right.
\end{equation}
Using the maximum principle we can deduce that $$\Phi\geq 0,\qquad (x,t)\in\Omega_{\delta}\times[0,T].$$
Combining it with $\Phi(x_{0},t_{0})=0$, we obtain $\Phi_{n}(x_{0},t_{0})\geq0$
which gives the desired estimate (\ref{e3.8}),  thus complete the proof of the lemma.
\end{proof}
It follows from ({\ref{e3.15}) that  we can state the following result which is similar to Proposition 2.6 in \cite{SM}.
\begin{lemma}\label{l3.0}
Fix a smooth function $H: \Omega\times\tilde{\Omega}\rightarrow R$ and define $\varphi(x,t)=H(x,Du(x,t))$. Then
there holds
\begin{equation*}\label{e3.170}
|L\varphi|\leq C\sum g^{ii},\quad (x,t)\in \Omega_{T},
\end{equation*}
where $C$ is a positive constant depending on $h$, $H$,  $u_{0}$ and $\Omega$.
\end{lemma}
We can now proceed to do the $C^{2}$  estimates. The strategy is to
 bound the interior second derivative firstly.
\begin{lemma}\label{l3.5}
For each $t\in [0,T]$,  the following estimates hold:
\begin{equation}\label{e3.170}
\sup_{\Omega}\mid D^{2}u\mid\leq \max_{\partial\Omega\times[0,T]}\mid D^{2}u\mid+\max_{\bar{\Omega}}\mid D^{2}u_{0}\mid.
\end{equation}
\end{lemma}
\begin{proof}
Given any unit vector $\xi$,  by the concavity of F, $u_{\xi\xi}$ satisfies
$$\partial_{t}u_{\xi\xi}-g^{ij}\partial_{ij}u_{\xi\xi}
=\frac{\partial^{2}F}{\partial u_{ij}\partial u_{kl}}u_{ij\xi}u_{kl\xi}\leq 0.$$
Combining with the convexity of $u$, and using the maximum principle we obtain
\begin{equation*}
\begin{aligned}
0\leq |u_{\xi\xi}|=u_{\xi\xi}(x,t)&\leq \max_{\partial\Omega_{T}}u_{\xi\xi}\\
 &\leq\max_{\partial\Omega\times[0,T]}\mid D^{2}u\mid+\max_{\bar{\Omega}}\mid D^{2}u_{0}\mid.
\end{aligned}
\end{equation*}
Therefore the estimates (\ref{e3.170}) are satisfied.
\end{proof}
By tangentially differentiating the boundary
condition $h(Du)=0$ we have some second order derivative bounds on $\partial\Omega$, i.e,
\begin{equation}\label{e3.17}
u_{\beta\tau}=h_{p_{k}}(Du)u_{k\tau}=0.
\end{equation}
where $\tau$  denotes a tangential  vector.  The second order derivative estimates on
the boundary are controlled by $u_{\beta\tau}, u_{\beta\beta}, u_{\tau\tau}$.

In the following
 we give the arguments as in \cite{JU}. For $x\in\partial\Omega$, any unit vector $\xi$ can be written in terms of
 a tangential component $\tau(\xi)$ and a component in the direction $\beta$ by
 $$\xi=\tau(\xi)+\frac{\langle\nu,\xi\rangle}{\langle\beta,\nu\rangle}\beta,$$
 where
 $$\tau(\xi)=\xi-\langle\nu,\xi\rangle\nu-\frac{\langle\nu,\xi\rangle}{\langle\beta,\nu\rangle}\beta^{T}$$
 and
 $$\beta^{T}=\beta-\langle\beta,\nu\rangle\nu.$$
After a simple computation it yields
\begin{equation}\label{e3.19}
\begin{aligned}
|\tau(\xi)|^{2}&=1-(1-\frac{|\beta^{T}|^{2}}{\langle\beta,\nu\rangle^{2}})\langle\nu,\xi\rangle^{2}
-2\langle\nu,\xi\rangle\frac{\langle\beta^{T},\xi\rangle}{\langle\beta,\nu\rangle}\\
&\leq 1+C\langle\nu,\xi\rangle^{2}-2\langle\nu,\xi\rangle\frac{\langle\beta^{T},\xi\rangle}{\langle\beta,\nu\rangle}\\
&\leq C,
\end{aligned}
\end{equation}
where we use the strict obliqueness (\ref{e3.4}). Let $\tau\triangleq \frac{\tau(\xi)}{|\tau(\xi)|}$.
Then by (\ref{e3.17}) and (\ref{e3.4}), we obtain
\begin{equation}\label{e3.20}
\begin{aligned}
u_{\xi\xi}&=|\tau(\xi)|^{2}u_{\tau\tau}+2|\tau(\xi)|\frac{\langle\nu,\xi\rangle}{\langle\beta,\nu\rangle}u_{\beta\tau}+
\frac{\langle\nu,\xi\rangle^{2}}{\langle\beta,\nu\rangle^{2}}
u_{\beta\beta}\\
&=|\tau(\xi)|^{2}u_{\tau\tau}+\frac{\langle\nu,\xi\rangle^{2}}{\langle\beta,\nu\rangle^{2}}
u_{\beta\beta}\\
&\leq C(u_{\tau\tau}+u_{\beta\beta}).
\end{aligned}
\end{equation}
Along with specifying the boundary conditions
we can carry out the double derivative estimates in the direction $\beta$.
\begin{lemma}\label{l3.6}
For each $t\in [0,T]$, we have the estimates
 $$\max_{\partial\Omega} u_{\beta\beta}\leq C_{2}$$
 where $C_{2}>0$ depending only on $u_{0}$, $h$, $\tilde{h}$, $\Omega$.
\end{lemma}
\begin{proof}
We use the  barrier functions  for any $x_{0}\in\partial\Omega$ and thus
consider
$$\Psi\triangleq\pm h(Du)+C_{0}\tilde{h}+A|x-x_{0}|^{2}.$$
 As  the proof of (\ref{e3.16}),
we can find the constant $C_{0}$ and  $A$,  such that we have
\begin{equation*}\label{e3.21}
\left\{ \begin{aligned}L\Psi&\leq 0,\qquad
&(x,t)\in\Omega_{\delta}\times[0,T] , \\
\Psi&\geq 0,\qquad &(x,t)\in(\partial\Omega_{\delta}\times[0,T]\cup(\Omega_{\delta}\times\{t=0\}.
\end{aligned} \right.
\end{equation*}
By  the maximum principle  we get
$$\Psi\geq 0,\qquad (x,t)\in\Omega_{\delta}\times[0,T].$$
Combining it with $\Psi(x_{0},t_{0})=0$  and using Lemma \ref{l3.4} we obtain $\Psi_{\beta}(x_{0},t_{0})\geq 0$.
Furthermore we see from $\beta=(\frac{\partial h}{\partial p_{1}}, \frac{\partial h}{\partial p_{2}}, \cdots\frac{\partial h}{\partial p_{n}})$ that
$$\frac{\partial h}{\partial\beta}=\langle Dh(Du),\beta\rangle=\Sigma_{k,l}\frac{\partial h}{\partial p_{k}}u_{kl}\beta^{l}
=\Sigma_{k,l}\beta^{k}u_{kl}\beta^{l}=u_{\beta\beta}.$$
Then it shows that
$$|u_{\beta\beta}|=|\frac{\partial h}{\partial\beta}|\leq C_{2}.$$
\end{proof}
We shall obtain  the bound of double tangential derivative at the boundary.
\begin{lemma}\label{l3.7}
There exists a constant $C_{3}>0$ depending only on $u_{0}$, $h$, $\tilde{h}$, $\Omega$ such that
 $$\max_{\partial\Omega\times[0,T]}\max_{|\tau|=1, \langle\tau,\nu\rangle=0} u_{\tau\tau}\leq C_{3}.$$
\end{lemma}
\begin{proof}
Assume that $x_{0}\in\partial\Omega$, $t_{0}\in[0,T]$  and  $\nu=e_{n}$ to be the inner unit normal of $\partial\Omega$
at $x_{0}$. Such that
$$\max_{\partial\Omega\times[0,T]}\max_{|\tau|=1, \langle\tau,\nu\rangle=0} u_{\tau\tau}=u_{11}(x_{0},t_{0}).$$
For any $x\in\partial\Omega$, combining (\ref{e3.19}) with (\ref{e3.20}), we have
\begin{equation*}
\begin{aligned}
u_{\xi\xi}&=|\tau(\xi)|^{2}u_{\tau\tau}+\frac{\langle\nu,\xi\rangle}{\langle\beta,\nu\rangle}^{2}u_{\beta\beta}\\
&\leq (1+C\langle\nu,\xi\rangle^{2}-2\langle\nu,\xi\rangle\frac{\langle\beta^{T},\xi\rangle}{\langle\beta,\nu\rangle})u_{\tau\tau}
+\frac{\langle\nu,\xi\rangle^{2}}{\langle\beta,\nu\rangle^{2}}
u_{\beta\beta}\\
&\leq(1+C\langle\nu,\xi\rangle^{2}-2\langle\nu,\xi\rangle\frac{\langle\beta^{T},\xi\rangle}{\langle\beta,\nu\rangle})u_{11}(x_{0},x_{0})
+\frac{\langle\nu,\xi\rangle^{2}}{\langle\beta,\nu\rangle^{2}}u_{\beta\beta}
\end{aligned}
\end{equation*}
Without loss of generality, we assume that $u_{11}(x_{0},t_{0})\geq 1$, then by Lemma \ref{l3.4} and Lemma \ref{l3.6}
 we get
\begin{equation*}
\frac{u_{\xi\xi}}{u_{11}(x_{0},t_{0})}+2\langle\nu,\xi\rangle\frac{\langle\beta^{T},\xi\rangle}{\langle\beta,\nu\rangle}\leq 1+C\langle\nu,\xi\rangle^{2}
\end{equation*}
Let  $\xi=e_{1}$, then we have
\begin{equation*}
\frac{u_{11}}{u_{11}(x_{0},t_{0})}+2\langle\nu,e_{1}\rangle\frac{\langle\beta^{T},e_{1}\rangle}{\langle\beta,\nu\rangle}\leq 1+C\langle\nu,e_{1}\rangle^{2}
\end{equation*}
We see that the function
\begin{equation*}
w\triangleq A|x-x_{0}|^{2}-\frac{u_{11}}{u_{11}(x_{0},t_{0})}-2\langle\nu,e_{1}\rangle\frac{\langle\beta^{T},e_{1}\rangle}{\langle\beta,\nu\rangle}
+C\langle\nu,e_{1}\rangle^{2}+1
\end{equation*}
satisfies
$$w|_{\partial\Omega\times[0,T]}\geq 0, \quad w(x_{0},t_{0})=0.$$
As before, by (\ref{e3.170}) we can select the constant $A$  such that
$$w|_{(\partial B_{\delta}(x_{0})\cap\Omega)\times[0,T]}\geq 0.$$
Consider
$$-2\langle\nu,e_{1}\rangle\frac{\langle\beta^{T},e_{1}\rangle}{\langle\beta,\nu\rangle}+C\langle\nu,e_{1}\rangle^{2}+1$$
as a known function depending on $x$ and $Du$. Then by Lemma \ref{l3.0} we obtain
$$|L(-2\langle\nu,e_{1}\rangle\frac{\langle\beta^{T},e_{1}\rangle}{\langle\beta,\nu\rangle}
+C\langle\nu,e_{1}\rangle^{2}+1)|\leq C\sum g^{ii}.$$
Combining it with the proof of Lemma \ref{l3.5}  we have
$$Lw\leq C\sum g^{ii}.$$
As in the proof of Lemma \ref{l3.6}, we consider the function
$$\Upsilon\triangleq  w+C_{0}\tilde{h}.$$
A standard barrier argument shows that $$\Upsilon_{\beta}(x_{0},t_{0})\geq 0.$$
 A direct computation yields
\begin{equation}\label{e3.22}
u_{11\beta}\leq Cu_{11}(x_{0},t_{0}).
\end{equation}
On the other hand, differentiating $h(Du)$ twice in the direction $e_{1}$ at $(x_{0},t_{0}),$
 we have

$$h_{p_{k}}u_{k11}+h_{p_{k}p_{l}}u_{k1}u_{l1}=0.$$
The concavity of $h$ yields
$$h_{p_{k}}u_{k11}=-h_{p_{k}p_{l}}u_{k1}u_{l1}\geq \tilde{C}u_{11}(x_{0},t_{0})^{2}.$$
Combining it with $h_{p_{k}}u_{k11}=u_{11\beta}$, and using (\ref{e3.22}) we obtain
\begin{equation*}\label{e3.23}
\tilde{C}u_{11}(x_{0},t_{0})^{2}\leq C u_{11}(x_{0},t_{0})
\end{equation*}
Then we get the upper bound of $u_{11}(x_{0},t_{0})$ and the desired result follows.
\end{proof}
Using Lemma {\ref{l3.6}, {\ref{l3.7}, and (\ref{e3.20}), we obtain the $C^{2}$ a priori bound on the  boundary:
\begin{lemma}\label{l3.8}
There exists a constant $C_{4}>0$ depending on $h$, $\tilde{h}$, $u_{0}$ and $\Omega$, such that
\begin{equation*}\label{e3.24}
\sup_{\partial\Omega_{T}}|D^{2}u|\leq C_{4}.
\end{equation*}
\end{lemma}
Using it and Lemma \ref{l3.5}, the following conclusion is thus proven:
\begin{lemma}\label{l3.9}
There exists a constant $C_{5}>0$ depending on $h$, $\tilde{h}$ and $u_{0}$, $\Omega$ such that
\begin{equation*}\label{e3.25}
\sup_{\bar{\Omega}_{T},|\xi|=1}D_{ij}u\xi_{i}\xi_{j}\leq C_{5}.
\end{equation*}
\end{lemma}
By the Legendre transformation of $u$, using (\ref{e3.12}) and repeating the proof of the above lemmas we get
the forthcoming result:
\begin{lemma}\label{l3.10}
There exists a constant $C_{6}>0$ depending on $h$, $\Omega$, $\tilde{h}$, $\tilde{\Omega}$ and $u_{0}$, such that
\begin{equation}\label{e3.26}
\frac{1}{C_{6}}\leq\inf_{\bar{\Omega}_{T},|\xi|=1}D_{ij}u\xi_{i}\xi_{j}
\leq\sup_{\bar{\Omega}_{T},|\xi|=1}D_{ij}u\xi_{i}\xi_{j}\leq C_{6}.
\end{equation}
\end{lemma}

\section{proof of main result}
{\bf Proof of Theorem \ref{t1.1}:}
Now let $u_{0}$  be  a  $C^{2+\alpha}$ strictly convex function as in the conditions of Theorem \ref{t1.1}.
Combining Proposition \ref{p1.1} with Lemma \ref{l3.10}, $\forall T>0$, $\exists u\in C^{2+\alpha,1+\frac{\alpha}{2}}(\bar{\Omega}_{T})$  which satisfies (\ref{e1.1}) and (\ref{e3.26}).
Using the boundary condition, we have
\begin{equation}\label{e4.1}
|Du|\leq C_{7}
\end{equation}
where $C_{7}$ be a constant depending on $\Omega$ and $\tilde{\Omega}$.
By Theorem 1.1 in \cite{TY} and Schauder estimates for parabolic equations,
for any $\hat{\Omega}\subset\subset\Omega$ and $m\in\{0,1,2,\cdots,\}$, we have
\begin{equation*}\label{e4.2}
\sup_{x_{i}\in\hat{\Omega}, t_{i}\geq 1}
\frac{|D^{2+m}u(x_{1},t_{1})-D^{2+m}u(x_{2},t_{2})|}
{\max\{|x_{1}-x_{2}|^{\alpha},|t_{1}-t_{2}|^{\frac{\alpha}{2}}\}}\leq C_{8}
\end{equation*}
where $C_{8}$ is a constant depending on the known data and $dist(\partial\Omega, \hat{\Omega})$.
By $Arzel\grave{a}-Ascoli$ theorem, a diagonal sequence argument shows that for any
$\{t_{k}\}|_{k=1}^{+\infty}$ satisfying
$$\lim t_{k}=+\infty,$$
there exist a subsequence
$$\{t_{k_{j}}\}|_{j=1}^{+\infty}\subset\{t_{k}\}|_{k=1}^{+\infty}$$
and
$$\hat{u}\in C^{1+1}(\bar{\Omega})\cap C^{2+m}(\Omega).$$ Such that for any $\zeta<1$ and $x\in \Omega$, we have
$$\lim_{j\rightarrow+\infty}\|u(\cdot,t_{k_{j}})-\hat{u}(\cdot)\|_{C^{1+\zeta}(\bar{\Omega})}=0.$$
\begin{equation}\label{e4.30}
\lim_{j\rightarrow+\infty}D^{2+m}u(x,t_{k_{j}})=D^{2+m}\hat{u}(x),
\end{equation}
and $\hat{u}$ satisfies (\ref{e3.26}).
Then we get
\begin{equation*}\label{e4.3}
h(D\hat{u})|_{\partial\Omega}=0
\end{equation*}
and
\begin{equation*}\label{e4.4}
\lim_{j\rightarrow+\infty}F(D^{2}u(x,t_{k_{j}}))=F(D^{2}\hat{u}(x)),\quad x\in \Omega.
\end{equation*}
For each $l$, differentiating the equation (\ref{e1.1}) by $x_{l}$  yields
$$\partial_{t}u_{l}=g^{ij}\partial_{ij}u_{l}.$$
Integrating from $0$ to $t$ on both sides we obtain
$$u_{l}(x,t)-u_{l}(x,0)=\int_{0}^{t}g^{ij}\partial_{ij}u_{l}(x,\sigma)d\sigma.$$
Combining it with (\ref{e4.1}), (\ref{e4.30}), we have
\begin{equation*}\label{e4.5}
\lim_{t\rightarrow+\infty}g^{ij}\partial_{ij}u_{l}(x,t)=0,\quad x\in \Omega.
\end{equation*}
Using this fact along with (\ref{e4.30}), the following equation emerges:
\begin{equation*}\label{e4.6}
g^{ij}\partial_{ij}\hat{u}_{l}=0,\quad x\in \Omega,\quad l\in\{1,2,\cdots,n\}.
\end{equation*}
Specifically , it is claimed that
$$F(D^{2}\hat{u})=C_{9}, \quad x\in \Omega$$
for some constant $C_{9}$ and it follows from (\ref{e3.26}) that  $C_{9}>0$.
Then the claim of Theorem \ref{t1.1} follows from the above arguments. \qed

\vspace{5mm}
{\bf Acknowledgment:}The author was supported by NNSF of China (Grant
No. 11261008) and NNSF of Guangxi (Grant No. 2012GXNSFBA053009) and
 was very grateful to  Institute of Differential Geometry at Leibniz University
Hannover for the kind hospitality. The author would like to thank the referee for giving some valuable suggestions which
improved the paper.

\end{document}